\documentclass[reqno,oneside]{amsart}

\usepackage[english]{babel} 
\usepackage[latin1]{inputenc} 
\usepackage[T1]{fontenc} 
\usepackage{times}
\usepackage{amsmath,amsthm,amssymb,amsfonts,mathrsfs,latexsym} 
\usepackage{graphicx} 
\usepackage{verbatim} 
\usepackage[all]{xy} 
\usepackage[colorlinks, linkcolor=blue, citecolor=blue, urlcolor=blue, hypertexnames=true]{hyperref}  
\usepackage{multirow}

\newtheorem{thm1}{Theorem}
\newtheorem{thm}{Theorem}[section]
\newtheorem{lemA}{Lemma}
\newtheorem{lem}[thm]{Lemma}
\newtheorem{cor}[thm]{Corollary}

\theoremstyle{definition}

\newtheorem{rem}[thm]{Remark}
\newtheorem*{rem*}{Remark}

\newcommand{\mc}[1]{\mathcal{#1}}
\newcommand{\mf}[1]{\mathfrak{#1}}
\newcommand{\mr}[1]{\mathrm{#1}}
\newcommand{\C}{\mathbb{C}}
\newcommand{\N}{\mathbb{N}}
\newcommand{\R}{\mathbb{R}}
\renewcommand{\epsilon}{\varepsilon}
\renewcommand{\phi}{\varphi}
\renewcommand{\tilde}{\widetilde}
\renewcommand{\hat}{\widehat}
\DeclareMathOperator{\GL}{GL}
\DeclareMathOperator{\SL}{SL}
\DeclareMathOperator{\SO}{SO}
\DeclareMathOperator{\SU}{SU}
\DeclareMathOperator{\Sp}{Sp}
\renewcommand{\Re}{\mr{Re}}

\newcommand{\id}{\mr{id}}
\newcommand{\diag}{\mr{diag}}
\newcommand{\supp}{\mr{supp}}
\newcommand{\Span}{\mathrm{span }}

\DeclareFontFamily{U}{mathx}{\hyphenchar\font45}
\DeclareFontShape{U}{mathx}{m}{n}{
      <5> <6> <7> <8> <9> <10>
      <10.95> <12> <14.4> <17.28> <20.74> <24.88>
      mathx10
      }{}
\DeclareSymbolFont{mathx}{U}{mathx}{m}{n}
\DeclareFontSubstitution{U}{mathx}{m}{n}
\DeclareMathAccent{\widecheck}{0}{mathx}{"71}
\DeclareMathAccent{\wideparen}{0}{mathx}{"75}
\renewcommand\check\widecheck
\renewcommand\tilde\widetilde

\numberwithin{equation}{section}
\begin{document}
\selectlanguage{english} 

\begin{abstract}
It is proved that:
\begin{enumerate}
  \item The Fourier algebra $A(G)$ of a simple Lie group $G$ of real rank at least $2$ with finite center does not have a multiplier bounded approximate unit.
  \item The reduced $C^*$-algebra $C_r^*(\Gamma)$ of any lattice $\Gamma$ in a non-compact simple Lie group of real rank at least $2$ with finite center does not have the completely bounded approximation property.
\end{enumerate}
Hence, the results obtained by  J.~de~Canniere and the author \cite{MR784292} for $\SO_{e}(n,1)$, $n\geq 2$, and by M.~Cowling \cite{MR748862} for $\SU(n,1)$ do not generalize to simple Lie groups of real rank at least $2$.
\end{abstract}


\title{
Group C$^*$-algebras without the completely bounded approximation property
}
\author{Uffe Haagerup}
\date{\today}
\address{%
Department of Mathematics and Computer Science, University of Southern Denmark
\newline Campusvej 55, 5230 Odense M, Denmark}

\maketitle

\section*{Preamble by Alain Valette}
\begin{quotation}
In spring 2015, I contacted Uffe Haagerup about his manuscript ``Group C$^*$-algebras without the completely bounded approximation property'', written in 1986, and never published. I mentioned that Journal of Lie Theory might be a good place to publish it. Uffe Haagerup liked the idea and said he was willing to update the paper after the summer.

By a sad twist of fate, Uffe Haagerup tragically passed away in July 2015. After his untimely death, Maria Ramirez Solano volunteered to type the manuscript, and S\o ren Knudby accepted to write an introduction and update the bibliography. We heartily thank both of them for their help in making the manuscript available to the community. We also thank S\o ren Haagerup for giving us permission to publish his father's paper.
\end{quotation}

\section*{Introduction}
The Fourier algebra $A(G)$ of a locally compact group, introduced by Eymard \cite{MR0228628}, consists of the matrix coefficients of the regular representation. The Fourier algebra is the predual of the group von Neumann algebra $\mf M(G)$ generated by the regular representation. The multipliers $MA(G)$ of the Fourier algebra consists of those continuous functions $\phi$ on $G$ such that $\phi\psi\in A(G)$ for every $\psi\in A(G)$. One identifies $\phi$ with the corresponding operator $m_\phi$ on $A(G)$, and the multiplier norm $\|\phi\|_{MA}$ is the operator norm of $m_\phi$. If the transposed operator $m_\phi^*$ on $\mf M(G)$ is completely bounded, we say that $\phi$ is a completely bounded multiplier. The space of completely bounded multipliers is denoted $M_0A(G)$, and the completely bounded multiplier norm $\|\phi\|_{M_0A}$ is the completely bounded operator norm of $m_\phi^*$. We refer to the papers \cite{MR1120720,MR784292} for details.

Leptin \cite{MR0239002} showed that the Fourier algebra $A(G)$ has an approximate unit bounded in norm if and only if $G$ is amenable. In \cite{MR784292}, de~Canni{\`e}re and the author showed that the Fourier algebra of the non-amenable group $\SO_e(n,1)$, $n\geq 2$, admits an approximate unit bounded in the completely bounded multiplier norm. In \cite{MR748862}, Cowling obtained similar results for $\SU(n,1)$. In the first half of the paper, we show these results do not generalize to simple Lie groups of real rank at least $2$:

\begin{thm1}\label{thm:1}
The Fourier algebra $A(G)$ of a simple Lie group $G$ of real rank at least $2$ with finite center does not have an approximate unit bounded in multiplier norm.
\end{thm1}

The second half of the paper is concerned with applications to operator algebras. It is shown that the Fourier algebra $A(\Gamma)$ of a lattice $\Gamma$ in a second countable locally compact group $G$ has an approximate unit bounded in the completely bounded multiplier norm if and only if the Fourier algebra $A(G)$ of $G$ has such an approximate unit (Theorem~\ref{thm:lattice}). It is also shown that, for a discrete group $\Gamma$, the Fourier algebra $A(\Gamma)$ has an approximate unit bounded in the completely bounded multiplier norm if and only if the reduced group C$^*$-algebra $C^*_r(\Gamma)$ has the completely bounded approximation property, if and only if the group von Neumann algebra $\mf M(\Gamma)$ has the ($\sigma$-weak) completely bounded approximation property (Theorem~\ref{thm:operatoralgebras}). As a corollary, $C^*_r(\Gamma)$ does not have the completely bounded approximation property when $\Gamma$ is a lattice in a simple Lie group of real rank at least $2$ with finite center (Corollary~\ref{cor:27}).

The paper ends with an appendix containing a characterization, due to Bo\.zejko and Fendler, of completely bounded multipliers.

A preliminary version of this paper was completed in 1986 and has circulated among experts in the field. We now mention some of the developments related to this paper up until its publication.

In 1989, Cowling and the author \cite{MR996553} showed that the Fourier algebra of a simple Lie group with finite center and real rank 1 admits an approximate unit bounded in the completely bounded multiplier norm. This generalized the results of \cite{MR748862,MR784292}. The condition of finite center was subsequently removed by Hansen \cite{MR1079871}.

In 1996, Dorofaeff \cite{MR1245415,MR1418350} removed the finite center condition from Theorem~\ref{thm:1}, thus giving a complete characterization of simple Lie groups whose Fourier algebras admits multiplier bounded approximate units.

In 2005, Cowling, Dorofaeff, Seeger, and Wright \cite{MR2132866} extended the previous results to cover many non-simple Lie groups including all real algebraic linear groups.

In 2012, Ozawa gave a short proof of Theorem~\ref{thm:1} (see \cite{MR2914879} and \cite{K-wa-sl2}).

In 1994, a weaker approximation property (called the Approximation Property or simply the AP) than the one considered in Theorem~\ref{thm:1} was introduced by Kraus and the author in \cite{MR1220905}. In 2011--2013, it was shown by Lafforgue, de~la~Salle, de~Laat and the author \cite{MR3047470,MR2838352} that simple Lie groups $G$ of real rank at least $2$ with finite center do not even have the AP, thus improving Theorem~\ref{thm:1}. The finite center condition was subsequently removed in \cite{MR3453357}.

\section{Proof of Theorem~\ref{thm:1}}\label{sec:Lie-groups}
\subsection*{Reducing the problem to \texorpdfstring{$\SL(3,\R)$}{SL(3,R)} and \texorpdfstring{$\Sp(2,\R)$}{Sp(2,R)}}
In \cite{MR0259023}, Wang proved that any simple Lie group $G$ with finite center and real rank at least $2$ has Kazhdan's property $(T)$, by using the fact that all these groups contain a closed subgroup $G'$ with finite center and locally isomorphic to either $\SL(3,\R)$ or $\SO(2,3)$ (cf. \cite[Theorem 7.2]{MR0207712}). If $G'$ fails to have a multiplier bounded approximate unit for its Fourier algebra, so does $G$ (cf. \cite[Proposition~1.12]{MR784292}). It is elementary to check, that if $F$ is a finite normal subgroup of a locally compact group $H$, then $A(H)$ has a multiplier bounded approximate unit if and only if $A(H/F)$ has a multiplier bounded approximate unit. Thus if $G'$ and $G''$ are locally isomorphic simple Lie groups with finite center, then $A(G')$ has a multiplier bounded approximate unit if and only if $A(G'')$ has a multiplier bounded approximate unit. Hence, to prove Theorem~\ref{thm:1} it suffices to consider the two groups $\SL(3,\R)$ and $\SO(2,3)$ of real rank $2$. But $\SO(2,3)$ is locally isomorphic to $\Sp(2,\R)$ (cf. \cite[p.~519]{MR514561}). So we can as well work with $\SL(3,\R)$ and $\Sp(2,\R)$.

\subsection*{Case \texorpdfstring{$\SL(3,\R)$}{SL(3,R)}}
Consider the closed subgroup $G_0$ of $\SL(3,\R)$ consisting of the $3\times 3$-matrices of the form

$$
\left(
  \begin{array}{c@{}c}
    A & \begin{array}{c}
          b_1 \\
          b_2 
        \end{array}
     \\
    \begin{array}{cc}
       0 & 0 
     \end{array}
     & 1 \\
  \end{array}
\right)
\qquad A\in \SL(2,\R),\,\,\, b_1,b_2\in\R.
$$
Note that $G_0$ is isomorphic to the semidirect product $\SL(2,\R)\times_{\alpha} \R^2$, where $\alpha$ is the usual action of $\SL(2,\R)$ on $\R^2$. We will show that already $G_0$ fails to have a multiplier bounded approximate unit.
Put 
$$
K_0=\left(
  \begin{array}{c@{}c}
    \SO(2) & \begin{array}{c}
          0 \\
          0
        \end{array}
     \\
    \begin{array}{cc}
       0 & 0
     \end{array}
     & 1 \\
  \end{array}
\right),
$$
and let $N\subseteq G_0$ be the nilpotent group of upper triangular matrices with 1's in the diagonal ($N$ is the Heisenberg group). It is convenient to use the following coordinates for $N$:

$$
\left(
  \begin{array}{ccc}
    1 & x & z+\frac12 xy \\
    0 & 1 & y \\
    0 & 0 & 1 \\
  \end{array}
\right), \qquad (x,y,z)\in \R^3.
$$

\begin{lemA}\label{l:LemmaA}
  Let $\gamma$ be the diffeomorphism of $N$ given by
  $$\gamma(x,y,z)=\left(-x,\ -\frac{z}{\sqrt{1+x^2/4}},\ y\sqrt{1+x^2/4}\right). $$
  If $\phi\colon N\to\C$ is the restriction of a $K_0$-biinvariant function $\tilde \phi$ on $G_0$ to $N$ then $\phi=\phi\circ\gamma$.
\end{lemA}

\begin{proof}
  Define $u,v\in\SO(2)$ by

  $$u=\frac{1}{\sqrt{1+x^2/4}}\left(
                                       \begin{array}{cc}
                                         \frac x2 & 1 \\
                                         -1 & \frac x2 \\
                                       \end{array}
                                     \right)\qquad\text{and}\qquad v=-u.
                                     $$
  By direct computation one gets:
$$
\left(
  \begin{array}{cc}
    u & 0 \\
    0 & 1 \\
  \end{array}
\right)
\left(
  \begin{array}{ccc}
    1 & x & z+\frac12 xy \\
    0 & 1 & y \\
    0 & 0 & 1 \\
  \end{array}
\right)
\left(
  \begin{array}{cc}
    v & 0 \\
    0 & 1 \\
  \end{array}
\right)=
\left(
  \begin{array}{ccc}
    1 & x' & z'+\frac12 x'y' \\
    0 & 1 & y' \\
    0 & 0 & 1 \\
  \end{array}
\right)
$$            
where $(x',y',z')=\gamma(x,y,z)$. This proves the lemma.                       
\end{proof}

\begin{lemA}[``Failure of Fubini's Theorem'']\label{l:LemmaB}
  Let $\phi\in C_c^\infty(\R^2)$. Then the two double integrals
  $$I(\phi)=\int_{\R}\int_{\R}\frac{\phi(y,z)}{y^2-z^2}\,dz\,dy,\qquad J(\phi)=\int_{\R}\int_{\R}\frac{\phi(y,z)}{y^2-z^2}\,dy\,dz
  $$
  exist, when the inner integrals are taken in the principal value sense (exclude symmetric intervals around the zeroes of $y^2-z^2$, and let the length of the intervals go to zero). Moreover
  $$I(\phi)-J(\phi)=\pi^2\phi(0,0).$$
\end{lemA}

\begin{proof}[Proof (sketch).]
$I$, $J$ are the distributions on $\R^2$, whose Fourier transforms are given by the $L^\infty$-functions
\begin{eqnarray*}
  &&\hat I (t,u)= \pi^2 \chi_{\{u^2>t^2\}},\\
  &&\hat J(t,u) = -\pi^2\chi_{\{u^2<t^2\}},
\end{eqnarray*}
where $\chi$ denotes the characteristic function, and $D\to \hat D$ is the natural extension of the Fourier transform on $\R^2$, which on the $L^1$-functions is given by
$$\hat f(t,u)=\iint_{\R^2} e^{i(ty+uz)}f(y,z)\,dy\,dz.$$
Hence $\hat I-\hat J = \pi^2$, which is the Fourier transform of $\pi^2$ times the Dirac measure at $(0,0)$.
\end{proof}

\begin{lemA}\label{l:LemmaC}
The map
$$\phi\mapsto \iint_{\R^2} \frac{\phi(y,z)}{1+y^2-z^2}\,dy\,dz,\qquad \phi\in C_c^\infty(\R^2)$$
is a well-defined distribution $K$ on $\R^2$, independent of the order of integration, and $\hat K$ is the $L^\infty$-function
$$\hat K(t,u) =\left\{
                 \begin{array}{cc}
                   J_0(\sqrt{u^2-t^2}) & u^2>t^2 \\
                   0 & u^2<t^2 \\
                 \end{array}
               \right.
$$
where $J_0$ is the zero-order Bessel function.
\end{lemA}

\begin{proof}[Proof (sketch).]
  To compute $\hat K$, observe that $K$ is $\SO(1,1)$-invariant. Hence $\hat K$ can only depend on $u^2-t^2$,
  so it is sufficient to compute $\hat K(t,0)$ and $\hat K(0,u)$. But $\hat K(0,u)=J_0(u)$ follows from the formula
  $$J_0(u)=\int_{-1}^1 \frac 1{\sqrt{1-x^2}}\ e^{iux} \,dx.$$
\end{proof}

\begin{lemA}\label{l:LemmaD}
  Let $D$ be the distribution on $N$ given by
  $$D(\phi)= \int_{-\infty}^{\infty}\int_{-\infty}^{\infty}\int_{-\infty}^{\infty}\frac{\phi(x,y,z)}{(1+x^2/4)y^2-z^2}\,dz\,dy\,dx,\qquad \phi\in C_c^\infty(N)
  $$
  Then
  $$|D(\phi)| \le 2\pi^3 \|\phi\|_{A(N)} \quad \forall \phi \in C_c^\infty(N).
  $$
\end{lemA}  
  
\begin{proof}[Proof (sketch).]
  Note that one cannot permute the order of integrations $dz$, $dy$ (cf. Lemma \ref{l:LemmaB}).  However, it is not hard to check that $dy$ and $dx$ can be permuted. We have to prove that $D$ corresponds to an operator $T\in \mf M(N)$, the von Neumann algebra associated to the left regular representation of $N$, such that $\|T\|\le 2\pi^3$ (cf. \cite{MR0228628}).
  The Heisenberg group $N$ is of type $I$, and the infinite dimensional irreducible representations of $N$ are given by $(\rho_a)_{a\in\R\backslash\{0\}}$ acting on $L^2(\R)$ by:
  \begin{eqnarray*}
    &&(\rho_a(x,0,0)f)(t)= e^{iaxt}f(t)\\
    &&(\rho_a(0,y,0)f)(t)= f(t-y)\\
    &&(\rho_a(0,0,z)f)(t)= e^{iaz}f(t)
  \end{eqnarray*}
  where $f\in L^2(\R)$.
  The remaining irreducible representations are all one-dimensional, and they form a null set for the Plancherel measure on $\hat N$.
  For $f,g\in C_c^\infty(\R)$ the integral
    $$\int_{-\infty}^{\infty}\int_{-\infty}^{\infty}\int_{-\infty}^{\infty}\frac{\langle\rho_a(x,y,x)f),g\rangle }{(1+x^2/4)y^2-z^2}\,dz\,dy\,dx
  $$
  can be computed by permuting the order of integration $dy$, $dx$ and applying Lemma \ref{l:LemmaC}.
  After some reduction, one finds that the integral is equal to 
  $$\iint_{\R^2} \overline{g(s)}k(s,t)f(t) \,dt,
  $$
  where 
  $$k(s,t)=\left\{
              \begin{array}{cc}
                \frac{2\pi^2}{|s-t|}J_0(a\sqrt{-4st}) & st<0, \\
                0 & s t>0.
              \end{array}
            \right.
  $$

  Since $|J_0(x)|\le 1$ for all $x\in \R$,
	\begin{align}\label{e:starPage7}
	|k(s,t)|\leq 2\pi^2 K(s,t)
	\end{align}
	where
	\begin{align*}
	K(s,t)=\left\{
              \begin{array}{cc}
                \frac{1}{|s-t|} & st<0, \\
                0 & s t>0.
              \end{array}
            \right.
  \end{align*}

  Moreover, $K$ is the kernel of a bounded operator on $L^2(\R)$ of norm $\le \pi$ because it is the kernel of the operator 
  $$\frac{\pi}2 (UH-HU),$$
  where
  $$H(f)(s)=\frac{1}{\pi}\int_{-\infty}^\infty \frac{f(t)}{s-t}\,dt
  $$
  is the Hilbert transform (unitary on $L^2(\R)$),
  and $U$ is the unitary operator given by multiplication with $\mathrm{sign}(s)$, $s\in\R$.
  Therefore (\ref{e:starPage7}) implies that $k(s,t)$ is the kernel of an integral operator on $L^2(\R)$ with norm $\le 2\pi^3$. Hence $D$ corresponds to an operator in $\mf M(N)=\int_{\R\backslash\{0\}}^{\oplus} \rho_a(N)''\,da$ of norm $\le 2\pi^3$.
\end{proof}
  
\begin{lemA}\label{l:LemmaE}
  If $\phi\in C_c^\infty(N)$ and $\phi\circ \gamma=\phi$ then
  $$\left|\int_{-\infty}^{\infty} \frac{\phi(x,0,0)}{\sqrt{1+ x^2/4}}\,dx\right|\le 4\pi \|\phi\|_{A(N)},
  $$
  where $\gamma$ is the diffeomorphism of $N$ defined in Lemma \ref{l:LemmaA}.
\end{lemA}  
\begin{proof}
  Clearly,
  \begin{eqnarray*}
   D(\phi\circ \gamma)&=&\int_{-\infty}^\infty\int_{-\infty}^\infty\int_{-\infty}^\infty \frac{\phi\left(-x,-\frac{z}{\sqrt{1+x^2/4}},y\sqrt{1+x^2/4}\right)}{(1+x^2/4)y^2-z^2}\,dz\,dy\,dx\\
   &=&\int_{-\infty}^\infty\int_{-\infty}^\infty\int_{-\infty}^\infty \frac{\phi(x,y,z)}{z^2-(1+x^2/4)y^2}\,dy\,dz\,dx
  \end{eqnarray*}
  Hence, using Lemma \ref{l:LemmaB}
  \begin{eqnarray*}
  2D(\phi)&=&D(\phi)+D(\phi\circ\gamma)\\
  &=&\int_{-\infty}^\infty\int_{-\infty}^\infty\int_{-\infty}^\infty \frac{\phi(x,y,z)}{(1+x^2/4)y^2-z^2}\,dz\,dy\,dx 
	\\&&-\int_{-\infty}^\infty\int_{-\infty}^\infty\int_{-\infty}^\infty \frac{\phi(x,y,z)}{(1+x^2/4)y^2-z^2}\,dy\,dz\,dx\\
  &=&\pi^2\int_{-\infty}^{\infty} \frac{\phi(x,0,0)}{\sqrt{1+x^2/4}}\,dx.
  \end{eqnarray*}
  The lemma follows now from Lemma \ref{l:LemmaD}.
\end{proof}
  
\begin{proof}[Proof of Theorem~\ref{thm:1} for $\SL(3,\R)$.]
It is sufficient to show that the subgroup $G_0$ does not have a multiplier bounded approximate unit for $A(G_0)$:

Assume that there exists $\tilde\phi_n\in A(G_0)$, $n\in\N$, such that $\sup_n\|\tilde\phi_n\|_{MA(G_0)} <\infty$ and
$$
\lim_{n\to\infty} \|\phi\tilde\phi_n - \phi\|_{A(G_0)} = 0 \quad\text{for all } \phi\in A(G_0).
$$
Since $C_c^\infty(G_0)$ is dense in $A(G_0)$, we can choose the $\tilde\phi_n$-functions in $C_c^\infty(G_0)$, and by averaging by $K_0$-elements from left and right, we can also obtain that that $\tilde\phi_n$ is $K_0$-biinvariant.  Note that $\tilde \phi_n\to 1$ uniformly on compact subsets of $G_0$. Put 
$$\phi_n=\tilde \phi_n|_{N}.$$
  Then
  $$\|\phi_n\|_{A(N)}=\|\phi_n\|_{MA(N)} \le \|\tilde \phi_n\|_{MA(G_0)}.$$
The first equality holds because $N$ is amenable. Thus $\sup_n \|\phi_n\|_{A(N)} <\infty$, and by Lemma \ref{l:LemmaA},  $\phi_n =\phi_n\circ\gamma$ for all $n\in\N$. By using Lemma \ref{l:LemmaE} to $\phi_n^2$ we get 
$$
\left|\int_{-\infty}^\infty \frac{\phi_n(x,0,0)^2}{\sqrt{1+x^2/4}}\,dx\right| \le 4\pi \|\phi_n\|_{A(N)}^2.
$$
Thus by Fatou's Lemma,

$$
\liminf_{n\to\infty} \|\phi_n\|_{A(N)}^2 \ge \frac{1}{4\pi} \int_{-\infty}^{\infty} \frac{dx}{\sqrt{1+x^2/4}}=\infty.
$$
This gives a contradiction.
\end{proof}

\subsection*{Case \texorpdfstring{$\Sp(2,\R)$}{Sp(2,R)}}
The group $\Sp(2,\R)$ is the set of $g\in \GL(4,\R)$ that leaves invariant the exterior form
$$x_1x_3-x_3x_1+x_2x_4-x_4x_2,$$
(cf. \cite[p.~445]{MR514561}). For our purpose, it is convenient to permute the third and fourth coordinate in $\R^4$,
so we will instead consider the group $G\cong\Sp(2,\R)$ of invertible $4\times 4$-matrices that leave invariant the form:
$$x_1x_4-x_4x_1+x_2x_3-x_3x_2.$$
$G$ is a 10 dimensional connected Lie group with Lie algebra
\begin{align*}
\mf g=\left\{(a_{ij})_{i,j=1,\ldots,4} \middle\vert
\begin{array}{l}
a_{11}+a_{44}=a_{22}+a_{33}=a_{12}+a_{34}=a_{21}+a_{43}=0, \\
a_{13}=a_{24},\ a_{31}=a_{42}
\end{array}\right\}.
\end{align*}
Let $\mf g_0\subseteq \mf g$ be the Lie algebra
$$\mf g_0=\left\{
\left(
  \begin{array}{cccc}
    0 & -c_2 & c_1 & c_3 \\
    0 & -b_2 & b_1 & c_1 \\
    0 & b_3 & b_2 & c_2 \\
    0 & 0 & 0 & 0 \\
  \end{array}
\right)
\middle\vert\ b_1,b_2,b_3,c_1,c_2,c_3\in \R
\right\}.
$$
Then $\mf g_0=\mf g_1\oplus_s \mf g_2$ (semidirect sum), where $\mf g_1\cong\mf{sl}(2,\R)$  is the Lie algebra corresponding to $c_1=c_2=c_3=0$, and $\mf g_2$ is the Lie algebra corresponding to $b_1=b_2=b_3 = 0$. (Note that $\mf g_2$ is isomorphic to the Lie algebra of the three dimensional Heisenberg group). Hence $\exp(\mf g_0)$ generates a closed subgroup $G_0$ of $G$, namely, the semidirect product of
$$
G_1=\left\{
\left(
  \begin{array}{c@{}c@{}c}
    1 & \begin{array}{cc}
          0 & 0 
        \end{array}
     & 0 \\
   \begin{array}{c}
          0 \\
          0
        \end{array} & a &  \begin{array}{c}
          0 \\
          0
        \end{array}\\
    0 & \begin{array}{cc}
          0 & 0
        \end{array} & 1 \\
  \end{array}
\right)
\middle\vert\
a\in \SL(2,\R)
\right\},
$$
and 
$$
G_2=\exp(\mf g_2)=\left\{
\left(
  \begin{array}{cccc}
    1 & -c_2 & c_1 & c_3 \\
    0 & 1 & 0 & c_1 \\
    0 & 0 & 1 & c_2 \\
    0 & 0 & 0 & 1 \\
  \end{array}
\right)
\middle\vert\ c_1,c_2,c_3\in\R
\right\},
$$
  where the action $\alpha\colon G_1\to\mr{Aut}(G_2)$ is given by 
 $$
 \alpha(a)\left(
            \begin{array}{c}
              c_1 \\
              c_2 \\
              c_3 \\
            \end{array}
          \right)=
          \left(
  \begin{array}{c@{}c}
    a & \begin{array}{c}
          0 \\
          0
        \end{array}
     \\
    \begin{array}{cc}
       0 & 0
     \end{array}
     & 1 \\
  \end{array}
\right)
\left(
            \begin{array}{c}
              c_1 \\
              c_2 \\
              c_3 \\
            \end{array}
          \right).
 $$
Consider next the compact subgroup $K_0$ of $G_0$ given by
$$
K_0=\left(
  \begin{array}{c@{}c@{}c}
    1 & \begin{array}{cc}
          0 & 0
        \end{array}
     & 0 \\
   \begin{array}{c}
          0 \\
          0
        \end{array} & \SO(2) &  \begin{array}{c}
          0 \\
          0
        \end{array}\\
    0 & \begin{array}{cc}
          0 & 0
        \end{array} & 1 \\
  \end{array}
\right)
$$ 
and the nilpotent subgroup of $G_0$ given by
\begin{eqnarray*}
  N&=&\exp\left\{
\left(
  \begin{array}{cccc}
    0 & -c_2 & c_1 & c_3 \\
    0 & 0 & b_1 & c_1 \\
    0 & 0 & 0 & c_2 \\
    0 & 0 & 0 & 0 \\
  \end{array}
\right)
\middle\vert\ b_1,c_1,c_2,c_3\in\R
\right\}\\
&=&\left\{
\left(
  \begin{array}{cccc}
    1 & -c_2 & c_1-\frac12b_1c_2 & c_3-\frac16b_1c_2^2 \\
    0 & 1 & b_1 & c_1+\frac12b_1c_2 \\
    0 & 0 & 1 & c_2 \\
    0 & 0 & 0 & 1 \\
  \end{array}
\right)
\middle\vert\ b_1,c_1,c_2,c_3\in\R
\right\}.
\end{eqnarray*}
The group $N$ is isomorphic to the group $\Gamma_4$ considered by Dixmier in \cite{MR0095427}.

In the rest of this section, we will use the following coordinates for $N$:
$$
n(x,y,z,w)=\left(
  \begin{array}{cccc}
    1 & -y & z-\frac12xy & w \\
    0 & 1 & x & z+\frac12xy \\
    0 & 0 & 1 & y \\
    0 & 0 & 0 & 1 \\
  \end{array}
\right), \quad (x,y,z,w)\in\R^4.
$$

\begin{lemA}\label{l:LemmaF}
  Let $\gamma'$ be the diffeomorphism of $N$ given by 
  $$\gamma'(x,y,z,w)=\left(-x,-\frac{z}{\sqrt{1+x^2/4}},y\sqrt{1+x^2/4},w\right).$$
If $\phi\colon N\to\C$ is the restriction of a $K_0$-biinvariant function $\tilde \phi$ on $G_0$ to $N$, then $\phi=\phi\circ\gamma'$.  
\end{lemA}
\begin{proof}
  Let $u$, $v$ $\in \SO(2)$ be as in Lemma \ref{l:LemmaA}. Then
  $$
  \left(
  \begin{array}{c@{}c@{}c}
    1 & \begin{array}{cc}
          0 & 0
        \end{array}
     & 0 \\
   \begin{array}{c}
          0 \\
          0
        \end{array} & u &  \begin{array}{c}
          0 \\
          0
        \end{array}\\
    0 & \begin{array}{cc}
          0 & 0
        \end{array} & 1 \\
  \end{array}
\right)
n(x,y,z,w)
\left(
  \begin{array}{c@{}c@{}c}
    1 & \begin{array}{cc}
          0 & 0
        \end{array}
     & 0 \\
   \begin{array}{c}
          0 \\
          0
        \end{array} & v &  \begin{array}{c}
          0 \\
          0
        \end{array}\\
    0 & \begin{array}{cc}
          0 & 0
        \end{array} & 1 \\
  \end{array}
\right)
=
n(\gamma'(x,y,z,w)).
  $$
\end{proof}
 
\begin{lemA}\label{l:LemmaG}
  Let $D'$ be the distribution on $N$ given by 
  $$D'(\phi)= \int_{-\infty}^{\infty}\int_{-\infty}^{\infty}\int_{-\infty}^{\infty}\frac{\phi(x,y,z,0)}{(1+x^2/4)y^2-z^2}\,dz\,dy\,dx,\qquad \phi\in C_c^\infty(N).
  $$
Then
$$
|D'(\phi)|\le 2\pi^3\|\phi\|_{A(N)}\qquad \text{for all } \phi\in C_c^\infty(N).
$$
\end{lemA}
 
\begin{proof}[Proof (sketch).]
  $N$ is the semidirect product of the abelian subgroup 
  $$
  B=\left\{
  \left(
    \begin{array}{cccc}
      1 & 0 & z & w \\
      0 & 1 & x & z \\
      0 & 0 & 1 & 0 \\
      0 & 0 & 0 & 1 \\
    \end{array}
  \right)
  \middle\vert\ (x,z,w)\in\R^3
  \right\},
  $$
isomorphic to $\R^3$, and the one-parameter group 
  $$
  C=\left\{
  \left(
    \begin{array}{cccc}
      1 & -y & 0 & 0 \\
      0 & 1 & 0 & 0 \\
      0 & 0 & 1 & y \\
      0 & 0 & 0 & 1 \\
    \end{array}
  \right)
  \middle\vert\ y\in\R 
  \right\},
  $$ 
i.e. $N\cong\R^3 \times_\theta \R$ where the action $\theta$ is given by
$$
\theta_y(x,z,w)=(x,z-yx,w-2yz+y^2x).
$$
Let $(s,u,v)$ be the dual coordinates to $(x,z,w)$.
The transpose of $\theta_y$ is
$$
\hat\theta_y(s,u,v)=(s-yu+y^2v,u-2yv,v).
$$  
The irreducible representation of $N$ can now be obtained by the Mackey machine (cf. \cite{MR0044536}, \cite{MR0095427}):
The orbits for $\hat\theta_y$ are
\begin{enumerate}
  \item parabolas\qquad\quad\, ($v\ne0$),
  \item straight lines\qquad ($v=0,u\ne0$),
  \item single points \qquad ($v=0,u=0$).
\end{enumerate}
Since $\{(s,u,v)\,|\, v=0\}$ is a null set in $\R^3$, the first type of orbits gives sufficiently many irreducible representations to disintegrate the regular representation of $N$.
Let $\rho_{a,b}$ be the irreducible representation coming from the parabolic orbit with vertex $(b,0,a)$, $a\in\R\backslash\{0\}$, $b\in \R$.
Then, $\rho_{a,b}$ can be realized on $L^2(\R)$ as follows:
\begin{eqnarray*}
  &&(\rho_{a,b}(x,0,0,0))(t)=e^{i(at^2+b)x} f(t)\\
  &&(\rho_{a,b}(0,y,0,0))(t)=f(t-y)\\
  &&(\rho_{a,b}(0,0,z,0))(t)=e^{i2atz} f(t)\\
  &&(\rho_{a,b}(0,0,0,w))(t)=e^{iaw} f(t),
\end{eqnarray*}
where $f\in L^2(\R)$.
A computation similar to the one in Lemma \ref{l:LemmaD} gives now,
$$
\int_{-\infty}^{\infty}\int_{-\infty}^{\infty}\int_{-\infty}^{\infty}\frac{\langle \rho_{a,b}(x,y,z,0)f,g\rangle}{(1+x^2/4)y^2-z^2}\,dz\,dy\,dx=
\iint_{\R^2} \overline{g(s)}k(s,t)f(t)\,ds\,dt,
$$
where
$$k(s,t) =\left\{
                 \begin{array}{cc}
                   \frac{2\pi^2}{|s-t|} J_0(\sqrt{-(as^2+b)(at^2+b)}) & (as^2+b)(at^2+b)<0 \\
                   0 & (as^2+b)(at^2+b)>0. \\
                 \end{array}
               \right.
$$
If $ab\ge0$ then $k(s,t)=0$ almost everywhere in $\R^2$.
If $ab<0$ we put $c=\sqrt{-b/a}$. Then
$$
|k(s,t)|\le 2\pi^2 K(s,t),
$$
where $K(s,t)=\frac{1}{|s-t|} \chi_{\{(s^2-c^2)(t^2-c^2)<0\}}.$
But $K(s,t)$ is the kernel of a bounded integral operator on $L^2(\R)$ of norm $\le \pi$, namely, the operator
$$
\frac{\pi}2(U_2 HU_1-U_1HU_2),
$$
where $H$ is the Hilbert transform, and $U_1$ and $U_2$ are the unitary multiplication operators on $L^2(\R)$ given by the functions sign$(t+c)$ and sign$(t-c)$, respectively. Hence $k(s,t)$ is the kernel of a bounded integral operator of norm $\le 2\pi^3$. Since
$$
\mf M(N)=\iint_{a\ne0, b\in\R}^{\oplus} \rho_{a,b}(N)''\,da\,db,
$$
it follows that $D'$ corresponds to an element in $\mf M(N)\cong A(N)^*$ of norm $\le2\pi^3$.
\end{proof}

\begin{proof}[Proof of Theorem \ref{thm:1} for $\Sp(2,\R)$]
Exactly as in the proof of Lemma \ref{l:LemmaE}, we get that if $\phi\in C_c^\infty(N)$ and $\phi\circ\gamma'=\phi$, then
$$
\left|\int_{-\infty}^{\infty}\frac{\phi(x,0,0,0)}{\sqrt{1+x^2/4}}\,dx\right|\le 4\pi \|\phi\|_{A(N)},
$$
and the proof for $\Sp(2,\R)$ can be completed by the same arguments as we used for $\SL(3,\R)$.
\end{proof}

\begin{rem*}
  It follows from the above proofs that 
  $$
  \R^2\times_{s} \SL(2,\R)\qquad\text{and}\qquad HG\times_{s} \SL(2,\R)
  $$
  (the semidirect product of the Heisenberg group $HG$ by $\SL(2,\R)$ described above) both
  fail to have a multiplier bounded approximate unit for their Fourier algebras, although the Fourier algebras of $\R^2$, $HG$ and $\SL(2,\R)$ all have multiplier bounded approximate units. 
  ($\R^2$ and $HG$ are amenable, for $\SL(2,\R)$ see \cite{MR784292}). 
\end{rem*}

\section{Results about Lattices in Lie groups}\label{s:Section2}
A lattice $\Gamma$ in a locally compact group $G$ is a closed discrete subgroup, for which $G/\Gamma$ has a bounded $G$-invariant measure. 
A locally compact group that admits a lattice is necessarily unimodular (cf. Definition 1.8 and Remark 1.9 in \cite{MR0507234}).

In the following, $\Gamma$ denotes a lattice in a second countable locally compact group $G$.
In this case, $\Gamma$ is countable and the quotient map $\rho\colon G\to G/\Gamma$ has a a Borel cross section.
Let $\Omega$ be the range of a Borel cross section. Then
$$
G=\bigcup_{\gamma\in\Gamma} \Omega\gamma \qquad\text{(disjoint union).}
$$
Let $\mu$ be the Haar measure on $G$.
Since $\Gamma$ is countable, $\mu(\Omega)>0$. Moreover, the quotient map $\rho$ is a bijection of $\Omega$ onto $G/\Gamma$, that carries $\mu|_{\Omega}$ onto a $G$-invariant measure on $G/\Gamma$. Thus by the assumption that $\Gamma$ is a lattice, $\mu(\Omega)<\infty$.
In the following we will assume that $\mu$ is normalized such that 
$$\mu(\Omega)=1.$$
Let $\mu_{\Gamma}$ be the counting measure on $\Gamma$.
For every bounded function $\phi$ on $\Gamma$, the function
$$
\hat\phi=\chi_{\Omega}*\phi\mu_{\Gamma}*\tilde\chi_{\Omega}
$$
is a well-defined bounded continuous function on $G$ because $\chi_{\Omega}\in L^1(G)$, $G$ is unimodular, and because $\chi_{\Omega}*\phi\mu_\Gamma$ is the bounded function given by
$$
(\chi_{\Omega}*\phi\mu_{\Gamma})(\omega\gamma)=\phi(\gamma),\qquad \gamma\in\Gamma,\,\,\omega\in\Omega.
$$
\begin{lem}\label{l:Lemma2.1}\mbox{}
  \begin{enumerate}
    \item If $\phi\in A(\Gamma)$ then $\hat\phi\in A(G)$ and
    $$ \|\hat\phi\|_{A(G)} \le \|\phi\|_{A(\Gamma)}
    $$
    \item If $\phi\in M_0 A(\Gamma)$ then $\hat\phi\in M_0A(G)$ and 
    $$\|\hat\phi\|_{M_0A(G)} \le \|\phi\|_{M_0A(\Gamma)}.
    $$
  \end{enumerate}
\end{lem}  
\begin{proof}
\mbox{}

(1) If $\phi\in A(\Gamma)$, then there exists $f,g\in\ell^2(\Gamma)$ such that 
$\phi=f*\tilde g$, $\|f\|_2\|g\|_2=\|\phi\|_{A(\Gamma)}$.
Put $f_1=\chi_\Omega*f\mu_\Gamma$, $g_1=\chi_\Omega*g\mu_\Gamma$. Then $f_1,g_1\in L^2(G)$,
$\|f_1\|_2\|g_1\|_2=\|f\|_2\|g\|_2=\|\phi\|_{A(\Gamma)}$, and 
$f_1*\tilde g_1=\hat \phi$.
This proves (1).

(2) Every  $z\in G$ has a unique decomposition 
$$z=\omega\gamma,\qquad \omega\in\Omega, \gamma\in\Gamma.$$
The $\gamma$-part in the decomposition will be denoted by $\gamma(z)$. For $x\in G$, we let $\tau_x\colon\Omega\to\Omega$ be the map given by $\tau_x(\omega)=\omega'$, where $\omega'$ is the $\Omega$-component of $xw$ in the decomposition
$$x\omega = \omega' \gamma',\qquad \omega'\in \Omega, \,\,\gamma'\in\Gamma.$$
Each $\tau_x$ is a Borel isomorphism of $\Omega$ because $\tau_x$ corresponds to left translation by $x$ in $G/\Gamma$ with the  Borel isomorphism 
$\Omega\to G/\Gamma$ given by the quotient map.
Since the latter Borel isomorphism carries $\mu|_\Omega$ to an invariant measure on $G/\Gamma$, it follows that $\mu|_\Omega$ is $\tau_x$-invariant.

We rewrite the function $\hat\phi=\chi_\Omega*\phi\mu_\Gamma*\tilde \chi_\Omega$ in a suitable way.
Observe first that 
$$
(\chi_\Omega * \phi\mu_\Gamma)(\omega\gamma)=\phi(\gamma), \qquad \omega\in\Omega, \gamma\in\Gamma,
$$
or equivalently, 
$$
(\chi_\Omega * \phi\mu_\Gamma)(x)=\phi(\gamma (x)), \qquad x\in G.
$$
Therefore,
\begin{align*}
  \hat\phi(x)&= \int_G \phi(\gamma(y))\chi_\Omega(x^{-1}y)\,d\mu(y)\\
  &= \int_{x\Omega}\phi(\gamma(y))\,dy\\
  &= \int_\Omega \phi(\gamma(x\omega))\,d\mu(\omega).
\end{align*}

For $x,y\in G$ and $\omega\in\Omega$:
\begin{align*}
  x\omega &= \tau_x(\omega)\gamma(x\omega)\\
  y\omega &= \tau_y(\omega)\gamma(y\omega).
\end{align*}
Thus,
\begin{align*}
  yx^{-1}\tau_x(\omega)&=y\omega(x\omega)^{-1}\tau_x(\omega)\\
  &=\tau_y(\omega)\gamma(y\omega)\gamma(x\omega)^{-1}.
\end{align*}
Since $\tau_x(\omega)\in\Omega$ and $\gamma(y\omega)\gamma(x\omega)^{-1}\in\Gamma$, it follows that 
$$
\gamma(yx^{-1}\tau_x(\omega))=\gamma(y\omega)\gamma(x\omega)^{-1}.
$$
Hence
$$
\int_\Omega \phi(\gamma(y\omega)\gamma(x\omega)^{-1})\,d\mu(\omega)= \int_\Omega \phi(\gamma(yx^{-1}\tau_x(\omega)))\,d\mu(\omega).
$$
However, since $d\mu(\omega)$ is invariant under $\tau_x$, the latter integral is equal to 
$$
\int_\Omega \phi(\gamma(yx^{-1}\omega))\,d\mu(\omega)=\hat\phi(yx^{-1}).
$$
Hence, $\forall x,y\in G$:
\begin{align}\label{e:starpage20}
	\hat\phi(yx^{-1})=\int_\Omega \phi(\gamma(y\omega)\gamma(x\omega)^{-1})\,d\mu(\omega).
\end{align}
We can now apply Bozejko-Fendler's result that $M_0A(G)$ coincides with the space of Herz-Shur multipliers on $G$ with same norm \cite{MR753889}:
Since $\Gamma$ is discrete, it implies, that there exist a Hilbert space $H$ and bounded maps $\xi,\eta$ from $\Gamma$ to $H$ such that 
$$
\phi(\gamma_2\gamma_1^{-1})=\langle \xi(\gamma_1),\eta(\gamma_2)\rangle,\qquad \gamma_1,\gamma_2\in\Gamma,
$$
and
$$
\|\xi\|_\infty\|\eta\|_\infty=\|\phi\|_{M_0 A(\Gamma)}. 
$$
This follows from Gilbert's characterization of Herz-Schur multipliers \cite{GilbertIII}.
Define now
$\hat\xi,\hat\eta\colon\Gamma\to L^2(\Omega,H,d\mu)$ by
\begin{eqnarray*}
  &&\hat \xi(x)(\omega)=\xi(\gamma(x\omega))\\
  &&\hat \eta(x)(\omega)=\eta(\gamma(x\omega)).
\end{eqnarray*}
Then $\hat\xi$, $\hat\eta$ are bounded Borel functions from $G$ to $L^2(\Omega,H,d\mu)$ and 
$$
\sup_{x\in G} \|\hat\xi(x)\|_2\le \|\xi\|_\infty,\qquad \sup_{x\in G}\|\hat\eta(x)\|_2\le \|\eta\|_\infty,
$$
and by (\ref{e:starpage20}),
$$
\hat\phi(yx^{-1})=\int_\Omega \langle\xi(\gamma(x\omega)),\eta(\gamma(y\omega))\rangle\,d\mu(\omega)=\langle \hat \xi(x), \hat\eta(y)\rangle.
$$
Since $\hat \phi$ is continuous, it implies that $\hat\phi\in M_0A(G)$ and
$$
\|\hat\phi\|_{M_0A(G)} \le \|\hat\xi\|_\infty\|\hat\eta\|_\infty\le \|\xi\|_\infty\|\eta\|_\infty=\|\phi\|_{M_0A(\Gamma)},
$$
(cf. proof of \cite[Proposition~1.1]{MR748862} or \cite{GilbertIII}).
This proves (2).
\end{proof}

\begin{lem}\label{l:Lemma2.2}
  Let $G$ be a locally compact group and let $k\ge 1$.
  Then the following conditions are equivalent:
  \begin{enumerate}
    \item There exists a net $(\phi_\alpha)$ in $A(G)$ such that $\sup_\alpha \|\phi_\alpha\|_{M_0A(G)}\le k$ and $\phi_\alpha\to 1$ in $\sigma(L^\infty,L^1)$-topology.
    \item There exists a net $(\phi_\alpha)$ in $A(G)$ such that $\sup_\alpha \|\phi_\alpha\|_{M_0A(G)}\le k$ and $\phi_\alpha\to 1$ uniformly on compact sets.
    \item There exists an approximate unit $(\phi_\alpha)$ for $A(G)$ such that $\sup_\alpha\|\phi_\alpha\|_{M_0A(G)}\le k$. 
  \end{enumerate}
  (By an approximate unit in $A(G)$, we just mean a net $(\phi_\alpha)$ such that $$\lim_\alpha \|\phi_\alpha\psi-\psi\|_{A(G)}=0 \qquad \forall \psi\in A(G).$$
  The net $(\phi_\alpha)$ will in general be unbounded in $A(G)$-norm.)
\end{lem}
\begin{proof}
  (3)$\implies$(2)$\implies$(1) is clear.\\
  (1)$\implies$(2): Assume that $(\phi_\alpha)$ satisfies (1) and put 
  $$\phi'_\alpha=h*\phi_\alpha,
  $$
  where $h\in C_c(G)_+$, $\int h \,d\mu=1$. Then
  $$
  \phi'_\alpha(x)=\int_G h(xy)\phi_\alpha(y^{-1})dy.
  $$
  Let $K\subseteq G$ be compact. Then the functions $h(x\ \cdot\ )$, $x\in K$ form a compact subset of $L^1(G)$. Since 
  $\phi_\alpha\to 1$ in $\sigma(L^\infty, L^1)$, and since $\sup_\alpha \|\phi_\alpha\|_\infty \le k$, the convergence is uniform on compact subsets of $L^1(G)$.
  Hence,
  $$
  \phi'_\alpha(x) = \langle\check\phi_\alpha, h(x\ \cdot\ )\rangle
  $$
  converges to $\langle 1, h(x\ \cdot\ )\rangle = (h*1)(x)=1$ uniformly on $K$. Moreover, $\phi'_\alpha$ is contained in the $\sigma(L^\infty,L^1)$-closed convex hull of left translates of $\phi_\alpha$, and since the unit ball in $M_0A(G)$ is $\sigma(L^\infty,L^1)$-closed (cf. \cite{MR784292}), we have
  $$
  \|\phi'_\alpha\|_{M_0A(G)}\le \|\phi_\alpha\|_{M_0A(G)}\le k.
  $$
  (2)$\implies$(3): Assume that $(\phi_\alpha)$ satisfies (2). Choose $h\in C_c(G)_+$, $\int_G h \, d\mu =1$, and put $\phi'_\alpha=h*\phi_\alpha$.
  As above, we have $\|\phi'_\alpha\|_{M_0A(G)}\le k$ for all $\alpha$.
  Let $\psi\in A(G)\cap C_c(G)$, and put 
  $$K=\supp(h),\qquad L=\supp(\psi).$$
  Moreover, set $g_\alpha=\phi_\alpha\chi_{K^{-1}L}$ (where $\chi_E$ is the characteristic function of a set $E$).
  Then for $x\in L$,
  \begin{equation}\label{e:ipage23}
    (h*\phi_\alpha)(x)=\int_K h(y)\phi_\alpha(y^{-1}x)\,dy = (h*g_\alpha)(x)
  \end{equation}
Similarly,
  \begin{equation}\label{e:iipage23}
    (h*1)(x)= (h*\chi_{K^{-1}L})(x),\qquad x\in L.
  \end{equation}
By the assumption on $\phi_\alpha$, $g_\alpha\to \chi_{K^{-1}L}$ uniformly on $K^{-1}L$. Moreover, $g_\alpha$ vanishes outside the compact set $K^{-1}L$. Hence,
$$\tilde g_a\to \tilde \chi_{K^{-1}L}, \qquad \text{in $L^2(G)$-norm,}$$
  and since $h\in L^2(G)$, this implies that 
  $$
  h*g_\alpha \to h*\chi_{K^{-1}L} \qquad \text{in $A(G)$-norm.}
  $$
Hence also
  $$
  (h*g_\alpha)\psi \to (h*\chi_{K^{-1}L})\psi \qquad \text{in $A(G)$-norm,}
  $$
which by (\ref{e:ipage23}) and (\ref{e:iipage23}) is equivalent to
  $$
  \phi'_\alpha\psi=(h*\phi_\alpha)\psi \to (h*1)\psi=\psi \qquad \text{in $A(G)$-norm.}
  $$
Since $A(G)\cap C_c(G)$ is dense in $A(G)$ and since 
$$
\sup_\alpha \|\phi'_\alpha\|_{MA(G)} \le \sup_{\alpha} \|\phi'_\alpha\|_{M_0 A(G)}<\infty,
$$
it follows that 
$$
\|\phi'_\alpha\psi-\psi\|_{A(G)}\to 0,\qquad \text{for all $\psi\in A(G)$.}
$$
\end{proof}

\begin{rem*}
  The proof of (2)$\implies$(3) in Lemma \ref{l:Lemma2.2} is due to Michael Cowling (private communication). It substitutes a previous proof of ours, that was valid only for Lie groups.
\end{rem*}

\begin{thm}\label{thm:lattice}
  Let $\Gamma$ be a lattice in a second countable locally compact group $G$, and let $k\in [1,\infty[$.
  then the following conditions are equivalent.
  \begin{enumerate}
    \item $A(G)$ has an approximate unit $(\phi_\alpha)$, such that $\|\phi_\alpha\|_{M_0A(G)} \le k$, for all $\alpha$.
    \item $A(\Gamma)$ has an approximate unit $(\psi_\alpha)$, such that $\|\psi_\alpha\|_{M_0A(\Gamma)} \le k$ for all $\alpha$.
  \end{enumerate}
\end{thm}

\begin{proof}
  (1)$\implies$(2) follows from \cite{MR784292}, and is valid for any closed subgroup $\Gamma\subseteq G$.
  
  (2)$\implies$(1): Assume that $(\psi_\alpha)\subseteq A(\Gamma)$ satisfies (2), and put $\hat\psi_\alpha=\chi_\Omega*\psi_\alpha\mu_\Gamma*\tilde\chi_\Omega$, where
  $\Omega$  is chosen as in Lemma \ref{l:Lemma2.1}. Then $\hat\psi_\alpha\in A(G)$ by Lemma \ref{l:Lemma2.1}.
  Since $\Gamma$ is discrete, the net $(\psi_\alpha)$ is uniformly bounded, and since $\psi_\alpha\to 1$ pointwise in $\Gamma$, also
  $$\psi_\alpha\to 1_\Gamma,\qquad \sigma(\ell^\infty(\Gamma),\ell^1(\Gamma)).$$
  It is easy to check that the map $\phi\mapsto\hat\phi=\chi_\Omega*\phi\mu_\Gamma*\tilde\chi_\Omega$ from $\ell^\infty(\Gamma)$ to $L^\infty(G)$ is the transpose of a bounded map from $L^1(G)$ to $\ell^1(\Gamma)$. Hence,
  $\hat\psi_\alpha\to 1_G$ in $\sigma(L^\infty,L^1)$-topology. Moreover, 
  $\|\hat\psi_\alpha\|_{M_0A(G)}\le \|\psi_\alpha\|_{M_A(\Gamma)}$ by Lemma \ref{l:Lemma2.1}.
  Hence (1)$\implies$(2) follows from (1)$\iff$(3) in Lemma \ref{l:Lemma2.2}.
\end{proof}

\begin{cor}\label{c:Corollary2.4}
  Every lattice $\Gamma$ in a simple Lie group of real rank at least $2$ fails to have a complete multiplier bounded approximate unit.
\end{cor}
\begin{proof}
  Follows from Theorem~\ref{thm:lattice} and Theorem \ref{thm:1}.
\end{proof}

Let $\Gamma$ be a discrete group, and let $\mf M(\Gamma)$ be the von Neumann algebra associated with the left regular representation of $\Gamma$ on $\ell^2(\Gamma)$, and let 
$$
\delta_x(y)=\left\{
   \begin{array}{ll}
     1, & x = y \\
     0, & x\neq y
   \end{array}
 \right.
$$
be the standard basis in $\ell^2(\Gamma)$. Put 
$$\mathrm{tr}(a)=\langle a\delta_e,\delta_e\rangle,\qquad a\in \mf M(\Gamma).$$
Then $\mathrm{tr}$ is a normal faithful trace on $\mf M(\Gamma)$.

\begin{lem}\label{l:lemma2.5}
  Let $T$ be a bounded map from $C_r^*(\Gamma)$ to itself or a bounded map from $\mf M(\Gamma)$ to itself. Let
  $$\phi_T(x)=\mathrm{tr}(\lambda(x)^*T\lambda(x)), \qquad x\in \Gamma.
  $$
  \begin{enumerate}
    \item If $T$ is completely bounded, then $\phi_T\in M_0A(G)$, and 
    $$\|\phi_T\|_{M_0 A(\Gamma)}\le \|T\|_{CB}.$$
    \item If $T$ is of finite rank, then $\phi_T\in \ell^2(\Gamma)$.
  \end{enumerate}
\end{lem}
\begin{proof}
(1): Since $\lambda\otimes\lambda$ is unitarily equivalent to $\lambda \otimes 1_{\ell^2(\Gamma)}$ where $1_{\ell^2(\Gamma)}$ is the trivial representation of $\Gamma$ on $\ell^2(\Gamma)$ (cf. \cite[13.11.3]{MR0458185} and \cite{MR0150241}),
there exists a normal *-isomorphism $\pi$ of $\mf M =\mf M(\Gamma)$ onto a von Neumann subalgebra $\mf N$ of $\mf M\otimes\mf M$ such that
$$
\pi(\lambda(x))=\lambda(x)\otimes\lambda(x),\qquad x\in G.
$$
Let $\epsilon$ be the normal conditional expectation of $\mf M\otimes \mf M$ onto $\mf N$ that leaves $\mathrm{tr}\otimes\mathrm{tr}$ invariant.
Then, since $\varepsilon$ is orthogonal with respect to the inner product given by $\mathrm{tr}\otimes\mathrm{tr}$, one gets easily that
$$
\varepsilon(\lambda(x)\otimes\lambda(y))=\left\{
   \begin{array}{ll}
     \lambda(x)\otimes\lambda(x), & x= y, \\
     0, & x\ne y.
   \end{array}
 \right.
$$
Put $\rho=\pi^{-1}\varepsilon$. Since
$$
C_r^*(\Gamma)=\overline{\Span\{\lambda(g)\mid g\in\Gamma\}}^{\text{norm}}
$$
one has
$$
\pi(C_r^*(\Gamma))\subseteq C_r^*(\Gamma)\otimes C_r^*(\Gamma),
$$
and
$$
\rho(C_r^*(\Gamma)\otimes C_r^*(\Gamma))\subseteq C_r^*(\Gamma).
$$
Hence, if $T$ is completely bounded of $C_r^*(\Gamma)$ into itself,
$$
S=\rho\circ (T\otimes \id_{C_r^*(\Gamma)})\circ\pi
$$
is a well-defined completely bounded map on $C_r^*(G)$ and $\|S\|_{CB}\le \|T\|_{CB}$.
Now, 
$$T\lambda(x)=\sum_{y\in G} c_{x,y}\lambda(y),$$
where $c_{x,y}\in \C$ and the sum is convergent in the $\|\ \|_2$-norm associated with $\mathrm{tr}$. Hence
$$
(T\otimes\id_{C_r^*(\Gamma)})\pi(\lambda(x))=\sum_{y\in G} c_{x,y} \lambda(y)\otimes \lambda(x).
$$
By the orthogonality property of $\varepsilon$ we have
$$
\varepsilon(T\otimes\id_{C_r^*(\Gamma)})\pi(\lambda(x))=c_{x,x}\lambda(x)\otimes\lambda(x).
$$ 
Hence
$$
S\lambda(x) = c_{x,x}\lambda(x), \qquad x\in G.
$$ 
But 
$$
c_{x,x} = \langle T\lambda(x),\lambda(x)\rangle_{\mathrm{tr}} = \phi_T(x).
$$ 
This shows that (with the notation of \cite{MR784292}) $S=m_{\phi_T}^*$, and hence $\phi_T\in M_0A(\Gamma)$ and $\|\phi_T\|_{M_0A(G)}=\|S\|_{CB}\leq\|T\|_{CB}$, because $\pi$, $\pi^{-1}$ and $\varepsilon$ are completely positive.
If instead $T$ is a completely bounded map on $\mf M(\Gamma)$, then we let,
$$
S=\rho\circ(T\otimes\id_{\mf M(\Gamma)})\circ \pi,
$$
and the same proof applies (cf. \cite[Section~1]{MR784292}).

(2): It is sufficient to consider the rank one case:
A rank one map on $C_r^*(\Gamma)$ is of the form
$$
T(a)=f(a)b\qquad \text{where }f\in C_r^*(\Gamma)^*, b\in C_r^*(\Gamma).
$$
Thus
$$
\phi_T(x)=\mathrm{tr}(\lambda(x)^*b)f(\lambda(x)).
$$
But $x\mapsto\mathrm{tr}(\lambda(x)^*b)$ is in $\ell^2(\Gamma)$, because $(\lambda(x))_{x\in \Gamma}$ is an orthonormal family in $L^2(C^*_r(\Gamma),\mathrm{tr})$, and $x\mapsto f(\lambda(x))$ is bounded. This proves (2) in the $C_r^*(\Gamma)$-case. The $\mf M(\Gamma)$-case follows by the same arguments.
\end{proof}

\begin{thm}\label{thm:operatoralgebras}
  Let $\Gamma$ be a discrete group and let $k\ge 1$. Then the following three conditions are equivalent.
  \begin{enumerate}
    \item $A(\Gamma)$ has an approximate unit $(\phi_\alpha)$ , such that $\|\phi_\alpha\|_{M_0A(\Gamma)}\le k$ for all $\alpha$.
    \item There exists a net $(T_\alpha)$ of finite rank operators on $C_r^*(\Gamma)$ such that $\|T_\alpha\|_{CB}\le k$ for all $\alpha$, and such that
    $$\|T_\alpha x - x \| \to 0 \qquad \text{for all } x\in C_r^*(\Gamma).
    $$
    \item There exists a net $(T_\alpha)$ of $\sigma$-weakly continuous finite rank maps on $\mf M(\Gamma)$, such that $\|T_\alpha\|_{CB}\le k$ for all $\alpha$ and
        $$
        \langle \phi, T_\alpha x - x\rangle \to 0, \qquad \text{for all } x\in \mf M(\Gamma),\ \phi\in\mf M(\Gamma)_*.
        $$
  \end{enumerate}
\end{thm}

\begin{proof}
(1)$\implies$(2) and (1)$\implies$(3) follows from \cite{MR784292}.

(2)$\implies$(1): Assume $T_\alpha$  are finite rank operators of $C_r^*(\Gamma)$ into itself, such that $\sup_\alpha\|T_\alpha\|_{CB}\le k$ and
$$\|T_\alpha a -a \|\to 0\qquad\text{for all } a\in C_r^*(\Gamma).$$
Let $\phi_\alpha=\phi_{T_\alpha}$ be as defined in the preceding lemma. Then $\phi_\alpha\in \ell^2(\Gamma)\subseteq A(\Gamma)$ and
$$
\sup_\alpha \|\phi_\alpha\|_{M_0A(\Gamma)}< \infty,
$$
and since $\phi_\alpha(x)=\mathrm{tr}(\lambda(x)^*T_\alpha(\lambda(x)))\to 1$ for all $x\in\Gamma$:
\begin{equation}\label{e:starpage30}
  \|\phi_\alpha\psi-\psi\|_{A(\Gamma)}\to 0,
\end{equation}
for all $\psi\in A(\Gamma)$ with finite support, i.e. for all $\psi\in A(\Gamma)\cap C_c(\Gamma)$ which form a dense subset of $A(\Gamma)$. Using that $\|\phi_\alpha\|_{MA(\Gamma)}\le \|\phi_\alpha\|_{M_0A(\Gamma)}\le k$ for all $\alpha$, one gets that (\ref{e:starpage30}) holds for all $\psi\in A(\Gamma)$.

(3)$\implies$(1): Let $(T_\alpha)$ be a net satisfying the conditions in (3). Since the functional,
$$
a\mapsto\mathrm{tr}(\lambda(x)^*a), \quad a\in\mf M(\Gamma)
$$
is in $\mf M(\Gamma)_*$ ($x$ fixed) we have that 
$$\mathrm{tr}(\lambda(x)^*T_\alpha(\lambda(x)))\to 1,\qquad \forall x\in G.
$$
The proof can now be completed exactly as in (2)$\implies$(1) by use of Lemma \ref{l:lemma2.5}.
\end{proof}
\begin{cor}\label{cor:27}
  Let $\Gamma$ be a lattice in a simple Lie group of real rank at least $2$ with finite center. Then 
  \begin{enumerate}
    \item $C_r^*(\Gamma)$ does not have the completely bounded approximation property.
    \item $\mf M(\Gamma)$ does not have the ($\sigma$-weak) completely bounded approximation property.
  \end{enumerate}
\end{cor}

\section{Appendix: On completely bounded multipliers and Herz-Shur-Multipliers.}
In Section \ref{s:Section2}, we used Bozejko and Fendler's results \cite{MR753889} that $M_0A(G)$ coincides isometrically with the space $B_2(G)$ introduced by Herz in \cite{MR0425511}. Their result relies heavily on a characterization of $B_2(G)$ found by Gilbert (\cite[Theorem 4.7]{GilbertIII}). However, Gilbert's paper, has never been published, so for the convenience of the reader, we give below a self-contained proof of the result needed in Section \ref{s:Section2}. 

Let $a*b$  denote the Schur product $(a*b)_{ij}=a_{ij}b_{ij}$ of complex $n\times n$-matrices in $M_n(\C)$, and let $\|a\|$ be the $C^*$-norm on $M_n(\C)$, i.e. the operator norm of $a$ when $M_n(\C)$ acts on the Euclidean space $\ell^2(\{1,\ldots,n\})$. We let $\|\ \|_S$ denote the Schur multiplier norm on $M_n(\C)$, i.e.
$$
\|a\|_S=\sup\{\|a*b\|\mid b\in M_n(\C),\ \|b\|\le 1\}.
$$ 
Let $\le$ be the usual ordering in $M_n(\C)$ as a $C^*$-algebra, i.e.
$a\ge0$ iff $a=a^*$ and all eigenvalues of $a$ are nonnegative. It is well known that 
$$
a\ge 0,\ b\ge 0 \implies a*b\ge 0.
$$
Hence, if $a\ge 0$ then the operator $M_a\colon b\mapsto a*b$ is positive on the $C^*$-algebra $M_n(\C)$, and therefore,
$$
\|a\|_S=\|M_a(1)\|=
\left\|\left(
  \begin{array}{ccc}
    a_1 &  & 0 \\
     & \ddots &  \\
    0 &  & a_n \\
  \end{array}
\right)\right\|=\max\{a_1,\ldots,a_n\}
$$
for every $a\in M_n(\C)_+$ (see for instance \cite[Corollary~1]{MR0193530} or \cite[10.5.10]{MR1468230}).
The following lemma is a special case of \cite[Theorem~4.7]{GilbertIII}. The proof given here is inspired by \cite{MR733029}.
\begin{lem}[\cite{GilbertIII}]\label{l:Lemma3.1}
Let $a\in M_n(\C)$. The following three conditions are equivalent:
\begin{enumerate}
  \item $\|a\|_S\le 1$.
  \item $\exists b, c \in M_n(\C)_+$ such that $\left(
                                                  \begin{array}{cc}
                                                    b & a^* \\
                                                    a & c \\
                                                  \end{array}
                                                \right)\ge 0$
                                                and $b_{ii}\le 1$ $c_{ii}\le 1$, $i=1,\ldots,n$.
  \item There exist a Hilbert space $H$ and $2n$ vectors $\xi_1,\ldots,\xi_n$, $\eta_1$, \ldots, $\eta_n$ in the closed unit ball of $H$ such that
  $$
  a_{ij}=\langle \xi_i, \eta_j\rangle, \qquad i,j=1,\ldots,n.
  $$
\end{enumerate}
\end{lem}
\begin{proof}
(1)$\implies$(2): Consider the real subspace $E$ of $M_{2n}(\C)_{s.a.}$ consisting of the matrices of the form
$$\left(
  \begin{array}{cc}
    y & x^* \\
    x & z \\
  \end{array}
\right)
$$
where $y$, $z$ are self-adjoint diagonal matrices $y=\diag(y_1,\ldots,y_n)$, $z=\diag(z_1,\ldots,z_n)$. Let $\phi\colon E\to\C$ be the linear form given by 
\begin{equation}\label{e:starpage32}
  \phi\left(
  \begin{array}{cc}
    y & x^* \\
    x & z \\
  \end{array}
\right)
=
\sum_{i=1}^n (y_i+z_i)+2\Re\left(\sum_{i,j=1}^n x_{ij}a_{ij}\right).
\end{equation}
We will prove that $\phi\ge0$, i.e. $\phi(w)$ is non-negative on positive hermitian matrices in $E$.
Note that $1_{2n}\in E$. Since
$$
\phi(w)\ge0 \iff \phi(w+\varepsilon 1_{2n})\ge0,\qquad \forall \varepsilon >0,
$$
it is sufficient to check that $\phi(w)\ge 0$ for all strictly positive
$$
w=\left(
  \begin{array}{cc}
    y & x^* \\
    x & z \\
  \end{array}
\right)
$$
i.e. those $w\in E_+$ for which $y_i>0, z_i>0, i=1,\ldots,n$.
However, $w\ge0$ implies that 
$$
\left(
  \begin{array}{cc}
    1_n & y^{-1/2}x^*z^{-1/2} \\
    z^{-1/2}xy^{-1/2} & 1_n \\
  \end{array}
\right)\ge0
$$
which is equivalent to $\|z^{-1/2}xy^{-1/2}\|\le 1$\,\, ($C^*$-norm).
Since $\|a\|_S\le 1$, the matrix,
$$
e=(a_{ij}z_i^{-1/2}x_{ij} y_j^{-1/2})_{i,j=1,\ldots,n}
$$
has also $C^*$-norm $\le1$. Hence
$$
\left|\sum_{i,j=1}^n x_{ij}a_{ij} \right| = \left|\sum_{i,j=1}^n e_{ij}y_{j}^{1/2}z_{i}^{1/2}  \right|\le 
\Bigg(\sum_{j=1}^n y_j \Bigg)^{1/2}\Bigg(\sum_{i=1}^n  z_i  \Bigg)^{1/2}
$$
which implies that 
$$
\phi\left(
  \begin{array}{cc}
    y & x^* \\
    x & z \\
  \end{array}
\right)\ge 
\sum_{j=1}^n y_j + \sum_{i=1}^n z_i - 2\Bigg(\sum_{j=1}^n y_j \Bigg)^{1/2}\Bigg(\sum_{i=1}^n  z_i  \Bigg)^{1/2}\ge0.
$$
Hence $\phi$ is a positive functional on $E\subseteq M_{2n}(\C)$.

Let $\tilde\phi$ be a Hahn-Banach extension of $\phi$ to $M_{2n}(\C)_{s.a.}$. Then $\tilde\phi\ge0$,
because $\|\tilde\phi\|=\|\phi\|=\phi(1_{2n})$, and because 
$$
0\le w\le 1_{2n}\implies \tilde\phi(1_{2n}-w)\le \|\tilde \phi\| \implies \tilde\phi(w)\ge0.
$$ 
Hence, there exists a positive hermitian  matrix

$$(d_{ij})_{i,j=1,\ldots,2n}$$
such that
$$
\tilde \phi(w)=\sum_{i,j=1}^{2n} d_{ij}w_{ij}, \qquad w\in M_{2n}(\C)_{s.a.}.
$$
For diagonal matrices $w=\diag(y_1,\ldots,y_n,z_1,\ldots,z_n)$,
$\tilde\phi(w)=\sum_{i=1}^n(y_i+z_i)$ by (\ref{e:starpage32}). Therefore $d_{ii}=1$, $i=1,\ldots,2n$.
Moreover, if $w$ is of the form
$$
w=\left(
    \begin{array}{cc}
      0 & x^* \\
      x & 0 \\
    \end{array}
  \right),\qquad x\in M_n(\C).
$$  
then by (\ref{e:starpage32})
$$
2\Re\left(\sum_{i,j=1}^n a_{ij}x_{ij}  \right)=\sum_{i,j=1}^n d_{n+i,j}x_{ij}+\sum_{i,j=1}^n d_{i,n+j}(x^*)_{ij}= 2\Re\left(\sum_{i,j=1}^n d_{n+i,j} x_{ij} \right)
$$
for all $(x_{ij})\in M_n(\C)$. Hence $a_{ij}=d_{n+i,j}$, i.e.
$$
d=\left(
    \begin{array}{cc}
      b & a^* \\
      a & c \\
    \end{array}
  \right)
$$
where $b_{ii}=c_{ii}=1$, $i=1,\ldots,n$. This proves (2).

(2)$\implies$(3): Let 
$$
d=\left(
    \begin{array}{cc}
      b & a^* \\
      a & c \\
    \end{array}
  \right)
$$
be as in (2). Let $H=\ell^2(\{1,\ldots,2n\})$, and let $\eta_1,\ldots,\eta_n,\xi_1,\ldots,\xi_n$ be the row vectors of the operator $f=d^{1/2}$.
Since 
$$d_{ij}=\sum_k f_{ik} \overline{f_{jk}}$$
we get $a_{ij}=\langle\xi_i,\eta_j\rangle$, and $\|\xi_i\|^2=d_{n+i,n+i}\le 1$, $\|\eta_i\|^2=d_{ii}\le 1$, for $i=1\ldots, n$.

(3)$\implies$(2):
Let $\xi_i,\eta_j$ be as in (3), put 
$$
b_{ij}=\langle \eta_i,\eta_j\rangle, \qquad c_{ij}=\langle \xi_i,\xi_j\rangle \qquad i,j=1\ldots,n,
$$ 
and let 
$$
d=\left(
    \begin{array}{cc}
      b & a^* \\
      a & c \\
    \end{array}
  \right).
$$
If we let $\eta_{n+i}=\xi_i$, $i=1,\ldots,n$. Then
$$
d_{ij}=\langle \eta_i,\eta_j\rangle, \qquad i,j=1,\ldots,2n.
$$
This implies that $d$ is positive hermitian, because, for $\lambda_1,\ldots,\lambda_{2n}\in\C$,
$$
\sum_{i,j=1}^{2n} d_{ij} \lambda_i\overline{\lambda_j} = \left\|\sum_{i=1}^{2n} \lambda_i\eta_i  \right\|^2 \ge 0.
$$
(2)$\implies$(1): If $$
d=\left(
    \begin{array}{cc}
      b & a^* \\
      a & c \\
    \end{array}
  \right)\ge0
$$
and  $b_{ii}\le 1$, $c_{ii}\le 1$, then the map $e\mapsto d*e$, $e\in M_{2n}(\C)$ is positive and $\|d*e\|\le \max_{1\le i\le2n}{d_{ii}}\|e\|$. Taking $e$ of the form $e=\left(
    \begin{array}{cc}
      0 & 0 \\
      x & 0 \\
    \end{array}
  \right)$,
  $x\in M_n(\C)$, we see that $\|a*x\|\le \|x\|$, $x\in M_n(\C)$. This proves (2)$\implies$(1).
\end{proof}  

\begin{thm}[\cite{MR753889}]\label{t:Theorem3.2}
Let $\phi$ be a continuous function on a locally compact group $G$, and let $k\ge 0$. Then the following conditions are equivalent.
\begin{enumerate}
  \item $\phi\in M_0A(G)$ and $\|\phi\|_{M_0A(G)} \le k$.
  \item For every finite set $x_1,\ldots,x_n$ in $G$, 
  $$\|\phi(x_j^{-1}x_i)_{i,j=1,\ldots,n}\|_S\le k
  $$
  \item There exist a Hilbert space $H$ and two bounded maps $\xi,\eta\colon G\to H$ such that
  $$ 
  \phi(y^{-1}x)=\langle \xi(x), \eta(y) \rangle,\qquad \forall x,y\in G
  $$
  and
  $$
  (\sup_{x\in G} \|\xi(x)\|)(\sup_{y\in G}\|\eta(y)\|)\le k.
  $$
	\item There exist a Hilbert space $H$ and two bounded maps $\xi,\eta\colon G\to H$ as in (3) with the additional property that the coordinate functions $\xi_i$ and $\eta_i$ (with respect to any orthonormal basis $(e_i)_{i\in I}$ of $H$) are continuous and the families $\{\xi_i\}_{i\in I}$ and $\{\eta_i\}_{i\in I}$ are locally countable.
\end{enumerate}
\end{thm}

\begin{proof}
  (1)$\implies$(2) (cf. \cite{MR753889}): Let $x_1,\ldots,x_n\in G$ and let $(e_{ij})_{i,j=1}^n$ be the matrix units in $M_n(\C)$. Put
  $$
  f_{ij}=\lambda(x_i^{-1}x_j) \otimes e_{ij} \in \mf M(G)\otimes M_n(\C).
  $$
  It is elementary to check that $f_{ij}^*=f_{ji}$, $f_{ij}f_{kl}=\delta_{jk}f_{il}$ and $\sum_{i=1}^n f_{ii} =1$, so
  $\mf N=\Span_{i,j} \{f_{ij}\}$
  is a subalgebra of $\mf M(G)\otimes M_n(\C)$ $*$-isomorphic to $M_n(\C)$. Since $*$-isomorphisms preserve norms, 
  $$
  \left\|\sum_{i,j=1}^n \alpha_{ij} f_{ij}  \right\|=
  \left\|\sum_{i,j=1}^n \alpha_{ij} e_{ij}  \right\|,\qquad \alpha_{ij}\in\C.
  $$
  Let $M_\phi\colon \mf M(G)\to\mf M(G)$ be as in \cite[Section~1]{MR784292}.
  Then by (1),
  $$
  \|M_\phi\otimes i_n\|\le k,
  $$
  where $i_n$ is the identity on $M_n(\C)$. But, 
  $$
  M_\phi\lambda(x)=\phi(x)\lambda(x), \qquad x\in G.
  $$ 
  Therefore
  $$
  (M_\phi\otimes i_n)f_{ij}=\phi(x_i^{-1}x_j)f_{ij},
  $$
  and hence for $a_{ij}\in \C$,
  $$
  \left\|\sum_{i,j=1}^n \phi(x_i^{-1}x_j)\alpha_{ij}f_{ij}  \right\|\le k \left\| \sum_{i,j=1}^n \alpha_{ij} f_{ij}\right\|.
  $$
  Since $f_{ij}$ in this inequality can be exchanged by the standard matrix units $e_{ij}$ of $M_n(\C)$, it follows that the matrix $(\phi(x_i^{-1}x_j))_{i,j=1,\ldots,n}$ as well as its transpose has Schur multiplier norm $\le k$.
  
  (2)$\implies$(3): Let $(\mc F, \subseteq)$ be the family of all finite subsets of $G$ ordered by inclusion. Assuming (2), we can for each $F\in\mc F$ find a Hilbert space $H_F$ and bounded maps 
  $$
  \xi_F, \eta_F\colon F\to H_F.
  $$
  with
  $$
  \sup_{x\in F} \|\xi_F(x)\|\le k^{1/2},\qquad \sup_{x\in F}\|\eta_F(x)\|\le k^{1/2},
$$
  and
  $$
  \phi(y^{-1} x) = \langle \xi(x), \eta(y)\rangle,\qquad x,y\in F.
  $$
(3) follows now easily by a standard ultraproduct argument:
Let namely $\mc U$ be a cofinal ultrafilter on $(\mc F,\subseteq)$
(i.e. an ultrafilter that contains all sets of the form $\{F'\in\mc F\mid F'\supseteq F\}$, where $F\in\mc F$),
and let $H_{\mc U}$ be the ultraproduct of $(H_F)_{F\in\mc F}$ corresponding to $\mc U$.
  For $x\in G$, we let $\xi(x)$, and $\eta(x)$ be the elements in $H_{\mc U}$ with representing sequences 
  \begin{eqnarray*}
    &&\xi(x)=(\xi_F(x))_{F\in\mc F}\\
    &&\eta(x)=(\eta_F(x))_{F\in \mc F}
  \end{eqnarray*}
 where we set $\xi_F(x)=\eta_F(x)=0$ if $x\not\in F$. 
 Since for fixed $x,y\in G$,  $x,y\in F$ eventually,
 we have
 $$
 \langle \xi(x), \eta(y)\rangle = \lim_{\mc U} \langle \xi_F(x),\eta_F(y)\rangle = \phi(y^{-1}x),\quad x,y\in G.
 $$
  Moreover, 
 $$
  \sup_{x\in G} \|\xi(x)\|\le k^{1/2},\qquad \sup_{x\in G}\|\eta(x)\|\le k^{1/2},
$$ 
(3)$\implies$(4): Let $\xi,\eta\colon G\to H$ satisfy the conditions in (3).
Let $\xi'(x)=P\xi(x)$, where $P$ is the orthogonal projection on the closed linear span of 
$$
\{\eta(x)\mid x\in G \},
$$
and put $\eta'(y)=Q\eta(y)$, where $Q$ is the orthogonal projection on the closed linear span of $\{\xi'(x)\mid x\in G\}$.
Then
$$
\phi(y^{-1} x) =\langle \xi'(x), \eta(y)\rangle=\langle \xi'(x),\eta'(y)\rangle, \qquad x,y\in H,
$$
and both $\{\xi'(x)\}_{x\in G}$ and $\{\eta'(x)\}_{x\in G}$ are total sets in the Hilbert space $H'=Q(H)$.
There is therefore no loss of generality to assume, that
$\{\xi(x)\}_{x\in G}$ and $\{\eta(x)\}_{x\in G}$ are total in $H$.   
 Since $\phi$ is continuous, the map
 $$
 (x,y)\to \langle \xi(x), \eta(y)\rangle
 $$
is separately continuous, and by the uniform boundedness, and totality of $\{\xi(x)\}_{x\in G}$ and $\{\eta(x)\}_{x\in G}$, it follows that 
$\xi,\eta\colon G\to H$ are continuous from $G$ to $H$ with $\sigma(H,H^*)$-topology. 
Let $\{\xi_i(x)\}_{i\in I}$ and $\{\eta_i(x)\}_{i\in I}$ be the coordinates of $\xi(x)$ and $\eta(x)$ with respect to a fixed basis $(e_i)_{i\in I}$ in $H$.
Then $\xi_i$ and $\eta_i$ are continuous complex valued functions. 
Moreover, for any open relative compact set $U$ in $G$,
$$
\sum_{i\in I} \int_U |\xi_i(x)|^2\,dx \le \int_U \|\xi(x)\|^2\,dx <\infty.
$$
By the continuity of $\xi_i$ and the fact, that no non-empty open set has Haar measure zero, it follows that all except countably many of the $\xi_i$'s vanish on the set $U$, i.e. $\{\xi_i\}_{i\in I}$  is locally countable, and similarly, $\{\eta_i\}_{i\in I}$ is locally countable.

(4)$\implies$(1): For $f\in L^\infty(G)$, we let $m(f)$ denote the multiplication operator $g\mapsto fg$ on $L^2(G)$. For $x\in G$
$$
\lambda(x) m(f)\lambda(x)^{-1} = m (\gamma_x(f)),
$$ 
where 
$\gamma_x(f)(y)=f(x^{-1}y)$. Put now,
$$
a_i=m(\check \xi_i),\quad b_i=m(\check \eta_i),
$$
where $\check g(x)=g(x^{-1})$.
Then $\sum_{i\in I} a_i^*a_i$ and $\sum_{i\in I}b_i^*b_i$ are convergent and
$$\left\|\sum_{i\in I} a_i^* a_i\right\|^{1/2}\left\|\sum_{i\in I} b_i^* b_i\right\|^{1/2} \le \sup_{x\in G} \|\xi(x)\|\cdot \sup_{x\in G} \|\eta(x)\|\le k.$$
Thus, we can define a bounded $\sigma$-weakly continuous map $\Phi$ on $B(L^2(G))$ by
$$
\Phi(s)=\sum_{i\in I} b_i^* s a_i.
$$
Now
\begin{eqnarray*}
  \Phi(\lambda(x))&=&\sum_{i\in I} m(\check \eta_i)^*\lambda(x) m(\check \xi_i)\\
  &=&\sum_{i\in I} m(\check \eta_i)^* m(\gamma_x(\check \xi_i))\lambda(x)\\
  &=&m\left(\sum_{i\in I} \overline{\check \eta_i} \gamma_x(\check \xi_i)\right)\lambda(x).
\end{eqnarray*}
Here we have used that $\{\xi_i\}$ and $\{\eta_i\}$ are locally countable, so pointwise convergence of the sum implies 
$\sigma(L^\infty, L^1)$-convergence. But for $y\in G$,
$$
\sum_{i\in I} \overline{\check \eta_i(y)} (\gamma_x(\check \xi_i)(y))=\sum_{i\in I} \overline{\eta_i(y^{-1})}\xi_i(y^{-1}x)=\langle \xi(y^{-1}x),\eta(y^{-1})\rangle=
\phi(y(y^{-1}x))=\phi(x).
$$

Hence $\Phi(\lambda(x))=\phi(x)\lambda(x)$. This implies that 
$$
\Phi(\mf M(G))\subseteq\mf M(G).
$$
For $\tilde s\in B(L^2(G)) \otimes M_n(\C)$,
$$
(\Phi\otimes i_n)(\tilde s)=\sum_{i\in I} (b_i\otimes 1_n)^*\tilde s (a_i\otimes 1_n).
$$
Hence,
$$
\|\Phi\otimes i_n\|\le \left\|\sum_{i\in I} a_i^*a_i\right\|^{1/2}\left\|\sum_{i\in I} b_i^*b_i\right\|^{1/2}\le k.
$$
Thus $\|\Phi\|_{CB}\le k$, so by \cite[Proposition~1.2]{MR784292}
$$
\phi\in M_0A(G)\quad\text{and}\quad \|\phi\|_{M_0A(G)}\le k.
$$
\end{proof}

\begin{cor}[\cite{MR753889} and \cite{MR0425511}]\label{c:Corollary3.3}
  Let $G_d$ be the group $G$ with discrete topology.
  Then
  $$
  M_0A(G)= M_0A(G_d)\cap C(G),
  $$
  and the $M_0A$-norms on the two spaces coincide.
\end{cor}
\begin{proof}
  Immediate from (1)$\iff$(2).
\end{proof}
\begin{rem}
  It is not hard to see that the functions $\xi,\eta\colon G\to H$ in (3) of Theorem \ref{t:Theorem3.2} can actually be chosen to be continuous:
  
From the proof of (3)$\implies$(4) it follows that $\xi,\eta$ can be chosen such that they are continuous from $G$ to $H$ with respect to the $\sigma(H,H^*)$-topology on $H$, and such that for every open relative compact set $U\subseteq G$, $\xi(U)$ and $\eta(U)$ are contained in a separable subspace of $H$. 
  This implies that $\xi,\eta\in L^\infty(G, H)$. Let now $f\in C_c(G)$ with $\|f\|_2=1$, and define 
  $$
  \hat\xi(x),\hat\eta(x)\in L^2(G,H) \qquad \text{for all } x\in G
  $$
  by 
  \begin{eqnarray*}
    &&\hat\xi(x)(z)=f(z)\xi(zx)\qquad z\in G,\\
    && \hat \eta(x)(z)=f(z)\eta(zx)\qquad z\in G.
  \end{eqnarray*}
  Then it is clear that
  $$
  \sup_{x\in G}\|\hat\xi(x)\|_2\le \sup_{x\in G} \|\xi(x)\|,\qquad  \sup_{x\in G}\|\hat\eta(x)\|_2\le \sup_{x\in G} \|\eta(x)\|.
  $$
  Moreover, 
  \begin{eqnarray*}
    \langle \hat\xi(x),\hat\eta(y)\rangle &=&\int|f(z)|^2\langle \xi(zx),\eta(zy)\rangle dz\\
    &=&\int|f(z)|^2\phi(y^{-1}x) dz\\
&=&\phi(y^{-1}x).
  \end{eqnarray*}
  Finally, using the fact that the group of right translations $(R_x)_{x\in G}$ defined by
$$
(R_x h)(z) = \Delta_G^{1/2}(x) h(zx),\quad h\in L^2(G,H)
$$
acts norm-continuously on $L^2(G,H)$, it is not not hard to check that $x\mapsto\hat\xi(x)$ and $y\mapsto\hat\eta(y)$ are norm continuous from $G$ to $L^2(G,H)$.
\end{rem}

\end{document}